\documentclass[a4paper]{amsart}
\usepackage{amsfonts,amscd,amssymb,amsmath,amsthm,mathrsfs,multirow,lscape,xfrac}
\usepackage{verbatim}
\usepackage{changepage}
\usepackage{color}

\theoremstyle{plain}
\newtheorem{theorem}{Theorem}[section]

\newtheorem{conjecture}[theorem]{Conjecture}
\newtheorem{cor}[theorem]{Corollary}
\newtheorem{def-thm}[theorem]{Definition-Theorem}

\theoremstyle{definition}
\newtheorem{example}[theorem]{Example}

\theoremstyle{definition}

\newtheorem{remark}[theorem]{Remark}

\def\min{\mathop{\mathrm{min}}}

\def\ZZ{\mathbb Z}

\def\pp{\mathbb{P}}
\def\qq{\mathbb{Q}}
\def\rr{\mathbb{R}}

\def\iter#1#2{#1^{\circ #2}}
\def\id{\mathrm{id}}

\allowdisplaybreaks

\title[]{On higher dimensional Integrality and multiplicative dependence in semi-group algebraic dynamics}

\author{Jorge Mello}
\address{Department of Mathematics and Statistics, Oakland University, Michigan, USA}
\email{jorgedemellojr@oakland.edu}

\author{Yu Yasufuku}
\address{Department of Mathematics, School of Education, Waseda University, Tokyo, Japan}
\email{yasufuku@waseda.jp}

\begin{document}

\begin{abstract}
    We study multiplicative dependence of points in semi-group orbits in higher-dimensions.  More specifically, we show that the non-density of integral points in semi-group orbits implies sparsity of multiplicative dependence in orbits. This can be viewed as a semi-group dynamical and a higher-dimensional version of recent results by B\'{e}rczes--Ostafe--Shparlinski--Silverman, which in turn can be viewed as a generalization of theorems of Northcott and Siegel.  We also confirm that the non-density hypothesis of integral points in orbits is implied by Vojta's conjecture.
\end{abstract}

\maketitle

\section{Introduction}
Starting with Silverman \cite{S1}, integral points in orbits under rational maps have been studied extensively.  For example, \cite{HS} has studied how the number of integral points varies depending on various parameters, and \cite{CHT, GH} have also studied this problem over function fields and fields of bounded degree, to cite a few.  Moreover, the first author \cite{M2} treated a generalization to the setting of semi-group dynamics. The arithmetic dynamics of a single map is analogous to a rank-one abelian variety, and the full analogy with abelian varieties is in some sense the semi-group dynamics of multiple maps.  Further, several authors \cite{Y,Y2, Mat, GN, CSTZ, M3} have studied the higher-dimensional analogs.

Some of the integrality results in dimension 1 have implications on the multiplicative dependence of elements in an orbit of a point. Namely, in \cite{BOSS, OSSZ}, the authors prove finiteness results for algebraic numbers whose iterates by a rational function are multiplicatively dependent. For example, when $k$ is a number field, $f(x)\in k(x)$ has degree at least $2$ and is not a constant multiple of power maps, and $\Gamma \subset k^*$ is a finitely-generated group,
\[
\left\{(n, m, \alpha, u): \alpha\in \pp^1(k), u\in \Gamma, (f^{\circ n}(\alpha))^{r} =u (f^{\circ m}(\alpha))^{s}, n \gg m \right\}
\]
is finite for any fixed nonzero integers $r,s$, where we denote by $f^{\circ n}$ the $n$-th iterate of $f$. Their results were recently generalized to multiplicative dependence modulo approximate finitely generated groups \cite{BBGMOS}. These types of results can be viewed as generalizations of Siegel's theorem on the finiteness of elements with algebraic integer images under a rational function and also of Northcott's Theorem on the finiteness of preperiodic points.

In this paper, we generalize these results to multiplicative dependence under multiple maps on higher-dimensional varieties.  Unlike in dimension $1$, known integrality results are rather scarce and often based on conjectures, so our main result deals with multiplicative dependence for semi-group dynamics assuming that a certain Zariski-non-density result holds for integral points in orbits.  To state our results precisely, let $\phi_1, \ldots, \phi_\ell$ be self-maps. We denote by $\mathcal F$ the semi-group generated by $\phi_1, \ldots, \phi_\ell$ under composition, and by $\mathcal O_{\mathcal F}(P)$ the semi-group orbit of a point $P$
\[
\{\phi(P): \phi\in \mathcal F\}.
\]
On $\pp^N$, we can define the multiplicative structure by defining $[a_0:\cdots : a_N]\cdot [b_0:\cdots : b_N] = [a_0 b_0 : \cdots :a_N b_N]$, and we denote $[a_0^r:\cdots : a_N^r]$ by $[a_0:\cdots : a_N]^r$ for any $r\in \mathbb Z$ if none of the coordinates is zero.  Via the standard inclusion $\mathbb G_m^N \hookrightarrow \pp^N$ defined by $(a_1,\ldots, a_N) \mapsto [1:a_1:\cdots : a_N]$, the multiplicative structures on $\pp^N$ is compatible with that of $\mathbb G_m^N$.  Our first result generalizes some instances of \cite{BOSS} to higher dimension for multiple maps.

\begin{theorem}\label{thm:fggp}
Let $k$ be a number field, and $\Gamma$ be a finitely generated subgroup of $\mathbb{G}_m^N(k) \subset \pp^N(k)$. For $i=1, \ldots, \ell$, let  $\phi_i$ be an endomorphism of $\mathbb P^N$ of degree $d_i \ge 2$ defined over a number field $k$, and let $\mathcal F$ be the semi-group generated by $\phi_1, \ldots, \phi_\ell$ under composition.   Now, for $\epsilon \in [0,1)$, let us consider the following hypothesis:
\begin{quote}
There exists a nontrivial subdivisor $D$ of $(X_0 X_1 \cdots X_N = 0)$  such that
for all sufficiently large finite subset $S$ of places containing all archimedean ones,
\begin{equation}\label{eq:thesetfggp}\tag{Hyp${}_\epsilon$}
\bigcup_{P \in \pp^N(k)} \big\{\phi(P): \phi\in \mathcal F\setminus \{\id\}, \sum_{v\notin S} \lambda_v(D, \phi(P)) < \epsilon h(D, \phi(P))\big\}
\end{equation}
is contained in some Zariski-closed $Z_\epsilon = Z_{\epsilon, D, S} \neq \pp^N$.
\end{quote}
Then for any $c < 1$, the above hypothesis for some $\epsilon \ge \frac{1 + c}2$ implies that there exist finitely many points $P_1,\ldots, P_m$ such that
\begin{multline}\label{eq:thesetofthm}
\left \{\phi(P): P \in \pp^N(k), \phi(P)^{r} =u \cdot \psi(P)^{s}, \phi \in \mathcal F\setminus \{\id\}, \psi \in \mathcal F, u\in \Gamma, \right.\\
\left. \phi(P), \psi(P) \notin |D|, r, s\in \mathbb Z\setminus \{0\}, \left|\frac sr \right| \deg \psi \le c \deg \phi\right\}
\end{multline}
is contained in
\[
Z_\epsilon \cup \bigcup_{i=1}^m \mathcal O_{\mathcal F}(P_i).
\]
\end{theorem}

From the proof of Theorem \ref{thm:fggp}, it will be clear that we can generalize it to the setting of projective varieties, if we have endomorphisms $\phi_1, \ldots, \phi_\ell$ which are polarized by the same divisor; see Remark \ref{rem:polarized} for more details.  We will also mention in Remark \ref{rem:vectorrelation} how Theorem \ref{thm:fggp} can also be generalized to ``vector-type'' relations. In addition, we will comment in Remark \ref{rem:fixedrs} on what more one can say when we fix $r$ and $s$.

It will be interesting to know if there are examples where \eqref{eq:thesetofthm} is not contained in a Zariski-non-dense set but only contained in its union with finitely many orbits.  In other words, Theorem \ref{thm:fggp} does not rule out the possibility that there is a $P$ whose $\mathcal F$-orbit is Zariski-dense and there exist infinitely many $\phi\in \mathcal F$ for which $\phi(P)$ belongs to \eqref{eq:thesetofthm}. In the case of one-variable polynomials with $\Gamma$ being trivial, Young \cite{young} has shown that such a relation for infinitely many orbit points forces the maps to be special, but it is not clear if the same holds in higher-dimensions and/or when $\Gamma$ has some generators.

We now consider the case when the group relation occurs between an iterate and its post-composition with other maps.

\begin{theorem}\label{thm:fggp2}
Under the same setup  as Theorem \ref{thm:fggp}, for any $c < 1$, \eqref{eq:thesetfggp} of Theorem \ref{thm:fggp} for some $\epsilon \ge \frac{1 + c}2$ implies that for any $s, r\in \ZZ\setminus \{0\}$
\begin{multline}\label{eq:thesetofthm2}
\left \{(\phi(\psi(P)), \psi(P)): P\in \pp^N(k), \phi\in \mathcal F\setminus \{\id\}, \psi\in \mathcal F, \phi(\psi(P)) \notin Z_\epsilon,  \right. \\
\left. u\in \Gamma, \phi(\psi(P)), \psi(P) \notin |D|, \Big(\phi(\psi(P))\Big)^{r} = u \cdot \psi(P)^{s},  \left|\frac sr \right| \le c \deg \phi\right\}
\end{multline}
is finite.
\end{theorem}

Note that for the case of an orbit under a single map $f$, $\iter f n$ is equal to $\iter f m$ post-composed with $\iter f {(n-m)}$ when $n> m$.  As a result, we have the following corollary.
\begin{cor}
Under the same setup  as Theorem \ref{thm:fggp} but suppose that $\mathcal F$ is generated by just one map $f$ of degree $d$. Then for any $c < 1$, \eqref{eq:thesetfggp} of Theorem \ref{thm:fggp} for some $\epsilon \ge \frac{1 + c}2$ implies that for any $s, r\in \ZZ\setminus \{0\}$
\begin{multline*}
\left \{(\iter{f}{n} (P), \iter{f}{m} (P)): P\in \pp^N(k), n > m, \iter{f}{n} (P) \notin Z_\epsilon,  \right. \\
\left. u\in \Gamma, \iter{f}{n} (P), \iter{f}{m} (P) \notin |D|, \Big(\iter{f}{n} (P)\Big)^{r} = u \cdot \Big(\iter{f}{m} (P)\Big)^{s},  \left|\frac sr \right| \le c d^{n-m}\right\}
\end{multline*}
is finite.
\end{cor}

This corollary can be thought of as a higher-dimensional generalization of \cite{BOSS}, as explained in more detail in Remark \ref{rem:boss}.  The proof of Theorem \ref{thm:fggp} uses the usual height, while the proof of Theorem \ref{thm:fggp2} takes advantage of canonical heights to obtain more precise height bounds, as explained in further details in Remark \ref{rem:canonht}.

We now comment on the hypothesis \eqref{eq:thesetfggp}.  Since a set of $(D,S)$-integral points satisfy $\sum_{v\notin S} \lambda_v(D, P) \le C$ for some constant $C$, this hypothesis in particular shows Zariski-non-density of integral points in semi-group orbit.    Moreover, if the orbit is generic (i.e. any infinite subset is Zariski-dense), it follows that the set of integral points in the semi-group orbit is finite.  While \eqref{eq:thesetfggp} is known to be true for dimension $1$ by \cite{S1, HS}, it is still open in dimension at least $2$. The first author has obtained results \cite{M3} in this direction  using Diophantine approximation results for hypersurfaces by Evertse--Ferretti \cite{EF}, which in turn is based on Schmidt subspace theorem.  For a single map in higher-dimension, the second author \cite{Y} has obtained results in the direction of hypothesis \eqref{eq:thesetfggp} assuming a deep Diophantine conjecture by Vojta \cite{V2}, which can be thought of as a vast generalization of the Schmidt subspace theorem and a strengthening of results of Evertse--Ferretti.  Following these ideas, we show in Theorem \ref{thm:basedonvojta} situations where \eqref{eq:thesetfggp} holds for sufficiently small $\epsilon$,  assuming Vojta's Main Conjecture.  Since Theorems \ref{thm:fggp} and \ref{thm:fggp2} require knowing \eqref{eq:thesetfggp} for sufficiently large $\epsilon$, this is nowhere near satisfactory even under assuming Vojta's conjecture; we present it here merely to demonstrate that the hypothesis \eqref{eq:thesetfggp} at least seems reasonable. We provide a specific example of all of our results in Example \ref{example}.

\section{Preliminaries on heights}\label{sec:heights}

In this section, we set some notations and conventions around height functions.  For more details, see for example \cite{bomgub} and \cite{hinsil}.  Let $k$ be a number field, and let $M_k$ be the set of places.  For the unique archimedean place of $M_\qq$, we use the usual absolute value on $\rr$ as the normalized one.  For the non-archimedean place of $M_\qq$ corresponding to the prime $p$, we normalize the absolute value by defining the absolute value of $p$ to be $\frac 1p$.  We then normalize the absolute value corresponding to each place $v\in M_k$ by defining $|x|_v$ to be the $\frac{[k_v:\qq_v]}{[k:\qq]}$-th power of the absolute value in $v$ whose restriction to $\qq$ is a normalized absolute value on $\qq$.  We then define a Weil (global) height on the projective space $\pp^N$ by
\[
h([a_0:\cdots : a_N] ) = \sum_{v\in M_k} \log \max_i |a_i|_v
\]
for $[a_0:\cdots : a_N]\in \pp^N(k)$.  This becomes a well-defined function on $\pp^N(\overline{\qq})$.  For a nontrivial effective divisor $D$ on $\pp^N$ defined over $k$, we choose a homogeneous polynomial $F$ of degree $d$ with coefficients in the ring $\mathcal O_k$ of integers of $k$ so that $D$ is defined by $F$, and we define a local height function for each $v\in M_k$ by
\[
\lambda_v(D, [a_0:\cdots : a_N]) = \log \frac{(\max |a_i|_v)^d}{|F(a_0, \ldots, a_N)|_v}.
\]
This is a well-defined function on $(\pp^N \setminus |D|)(k)$.  For a general divisor $D$ on a projective variety $X$, we first write $D$ as the difference of two very ample Cartier divisors $D_1$ and $D_2$, where $D_i = \varphi_i^*(H_i)$ for some closed immersion $\varphi_i: X \longrightarrow \pp^{N_i}$ and $H_i$ is a hyperplane in $\pp^{N_i}$.  We then define Weil height to be
\[
h(D, P) = h(\varphi_1(P)) - h(\varphi_2(P)).
\]
Moreover, letting $s_{i,1},\ldots, s_{i, \ell_i}$ be a basis of global sections of the line bundle $\mathscr L(D_i)$ and choosing one rational section $s$ of $\mathscr L(D)$, we can define local height to be
\[
\lambda_v(D,P) = \max_m \min_\ell \log \left| \big(s_{1,m} \otimes (s_{2,\ell} \otimes s)^{-1}\big)(P) \right|_v,
 \]
where we evaluate at $P$ via the isomorphism $\mathscr L(D_1)\otimes (\mathscr L(D_2) \otimes \mathscr L(D))^{-1} \cong \mathscr O_X$.  This notion is well-defined in the sense that if we choose different parameters for the ample divisors or the global/rational sections, the two local height functions agree on all but finitely many places and even at those finitely many places, their difference is a bounded function.

The local and Weil height are functorial with respect to pullbacks, namely, whenever $\varphi: Y\longrightarrow X$ is a morphism of algebraic varieties, there exists a constant $C$ such that
\[
|h(\varphi^* D, P) -  h(D, \varphi(P))|\le C
\]
for $P\in Y(\overline \qq)$, and there exists a sequence of constants $\{C_v\}_{v\in M_k}$ such that $C_v = 0$ for all but finitely many $v$ and
\[
|\lambda_v(\varphi^*D, P) -  \lambda_v(D, \varphi(P))| \le C_v
\]
for all $P\in Y(k)\setminus \varphi^{-1}|D|$ and $v\in M_k$.  Moreover, there exists a constant $C$ such that
\begin{equation}\label{eq:htdecompose}
\left| h(D,P) -  \sum_{v\in M_k} \lambda_v(D,P) \right| \le C
\end{equation}
holds for $P\in (X\setminus |D|)(k)$. For an effective divisor $D$, there exists a sequence $\{C_v\}_{v\in M_k}$ such that $C_v = 0$ for all but finitely many $v$ and
\[
\lambda_v(D, P) \ge C_v
\]
holds for $P$ outside $|D|$ for each $v$; in particular, for a finite subset $S\subset M_k$, there exists a constant $C$ such that
\begin{equation}\label{eq:outsideSlocal}
\sum_{v\notin S} \lambda_v(D, P) \le h(D, P) + C
\end{equation}
for all $P$ outside $|D|$.

\section{Proofs of the main theorems}

\begin{proof}[Proof of Theorem \ref{thm:fggp}]
Since $\phi_i$'s are morphisms of degree $d_i$  on $\mathbb P^N$, there exists a constant $C$ such that
\[
\left|\frac{h(\phi_i(Q))}{d_i} - h(Q)\right| \le C
\]
holds for all $Q\in \mathbb P^N(\overline \qq)$ and for all $i=1,\ldots, \ell$. By an inductive argument using
\begin{align*}
&\left|\frac{h(\phi\circ\psi (Q))}{\deg(\phi) \cdot \deg(\psi)} - h(Q)\right|\\
&\le \frac 1{\deg(\psi)} \left|\frac{h(\phi\circ\psi (Q))}{\deg(\phi)} - h(\psi(Q))\right| + \left|\frac{h(\psi(Q))}{\deg(\psi)} - h(Q)\right|\\
\end{align*}
we must have
\begin{equation}\label{eq:usualhtineq}
\left| \frac{h(\phi(Q))}{\deg (\phi)} - h(Q) \right| \le \frac C{1- \frac 1{\min \deg(\phi_i)}} =: C_1
\end{equation}
for all $\phi\in \mathcal F$. Moreover, there exists a constant $C_2$ such that
\[
|h(D, Q) - (\deg D) h(Q)| \le C_2
\]
for all $Q\in \pp^N(\overline \qq)$.
Now, let $S$ be the set of places large enough so that hypothesis \eqref{eq:thesetfggp} is satisfied and the coordinates of all the generators are $S$-units.
This makes any coordinate of any element of $\Gamma$  into an $S$-unit (if nonzero).   Suppose that $\phi(P)$ is in the set \eqref{eq:thesetofthm}, $\phi(P) \notin Z$ for
a certain $\epsilon$ to be chosen later, and $\phi(P)$ lies outside $|D|$.  Then because of hypothesis \eqref{eq:thesetfggp} and the fact that any element in $\Gamma$ does not have any valuation outside of $S$,
\begin{align}
\epsilon & \Big( (\deg D) \big(\deg(\phi) h(P) - \deg(\phi) C_1\big) - C_2\Big)\notag \\
&\le \epsilon \Big( (\deg D) h(\phi(P)) - C_2\Big)\notag \\
&\le \epsilon h(D, \phi(P))\notag \\
&\le\sum_{v\notin S} \lambda_v(D, \phi(P)) \nonumber\\
&\leq \left|\frac sr\right| \sum_{v\notin S} \lambda_v(D, \psi(P)) \qquad (\because u\in \Gamma)\label{eq:fggp_rs}\\
& \leq \left|\frac sr\right| \left(h(D, \psi(P)) + C_3\right)  \qquad (\because \eqref{eq:outsideSlocal})\nonumber \\
&\leq \left|\frac sr\right| \left((\deg D) h(\psi(P)) + C_4\right)\nonumber\\
&\leq \left|\frac sr\right| \left( (\deg D) \Big(\deg (\psi) h(P) +  \deg (\psi) C_1\Big) + C_4\right).\nonumber
\end{align}
Therefore, we have
\begin{multline}
(\deg D) \Big(\epsilon \deg(\phi) - \left|\frac sr\right| \deg (\psi)\Big) h(P) \\
\le \epsilon \Big( (\deg D) \deg(\phi) C_1 + C_2\Big) + \left|\frac sr\right| \left((\deg D) \cdot \deg (\psi) C_1 + C_4\right). \label{eq:finalheight}
\end{multline}
By choosing $\epsilon$ to be at least $\frac {c+1}2$, it follows from \eqref{eq:thesetofthm} that the coefficient of $h(P)$ on the left-hand side of the \eqref{eq:finalheight} is at least $\deg D \cdot \frac {1-c}2 \cdot \deg (\phi)$, while the right-hand side can be bounded above by
\[
(\epsilon (\deg D) C_1 + c (\deg D) C_1 + c C_4) \deg(\phi) + \epsilon C_2.
\]
Thus, dividing both sides by $\deg(\phi)$, it follows that $h(P)$ must be bounded above by a constant, independent of $\phi$ or $\psi$.  Thus $P$ must come from a finite list by the Northcott property, and such a $\phi(P)$ must lie in finitely many orbits.
\end{proof}

\begin{remark}\label{rem:polarized}
This theorem extends to ``simultaneously polarizable endomorphisms'' on a projective variety $X$.  Namely, we fix an embedding $\iota: X\hookrightarrow \pp^N$, and suppose that there is an ample divisor $A$ such that $\phi_i^*(A)$ is linearly equivalent to $d_i A$ for some $d_i \ge 2$.  Then the same statement as the theorem holds if the hypothesis is satisfied after viewing the points in $\pp^N$ via $\iota$ and assuming the hypothesis via $\iota$.  The proof is exactly the same as above, and for this reason we have avoided using for example $h(A, Q) = (\deg A) h(Q)$, which is true without any adjustment by constants on projective space.
\end{remark}

\begin{remark}\label{rem:vectorrelation}
From the proof, it is evident that one can generalize Theorem \ref{thm:fggp} to more general multiplicative relations.  Namely,  letting $\vec r\in (\mathbb Z\setminus \{0\})^N$, we define $(a_1, \ldots, a_N)^{\vec r} = (a_1^{r_1}, \ldots, a_N^{r_N})$ if none of the coordinates is zero.  This does not in general extend to a multiplication on $\pp^N$, but whenever the orbit does not intersect the divisor $(X_0 \cdots X_N =0)$, this multiplication is well-defined via the inclusion $\mathbb G_m^N \hookrightarrow \pp^N$ we used above.  Since the only part of the argument affected by this change is \eqref{eq:fggp_rs}, we would obtain the same conclusion for the set
\begin{multline*}
\left \{\phi(P)\in \pp^n(k): \phi(P), \psi(P) \notin (X_0 \cdots X_N =0), \phi(P)^{\vec r} =u \psi(P)^{\vec s}, \right.\\
\left. u\in \Gamma, \vec r, \vec s\in (\mathbb Z\setminus \{0\})^{N}, c \frac{\deg \phi}{\deg \psi} \geq \max_{i: (X_i=0) \subset D} \frac {s_i}{r_i} \right\}.
\end{multline*}
\end{remark}

\begin{remark}\label{rem:fixedrs}
Rearranging \eqref{eq:finalheight}, we obtain
\[
(\deg D) \Big(\epsilon \deg(\phi) - \left|\frac sr\right| \deg (\psi)\Big) (h(P)- C_1) \le \epsilon C_2 +  \left|\frac sr\right| C_4.
\]
Our choice of $\epsilon$ as in the proof thus shows
\[
\deg D \cdot \frac {1-c}2 \cdot \deg (\phi) (h(P)- C_1) \le \epsilon C_2 +  \left|\frac sr\right| C_4.
\]
Therefore, either $h(P) \le C_1 + 1$ and thus $P$ comes from a finite set by the Northcott property, or we have
\[
\deg D \cdot \frac {1-c}2 \cdot \deg (\phi) < \epsilon C_2 +  \left|\frac sr\right| C_4.
\]
In the latter case, if we assume further that $r$ and $s$ are fixed (or at least if $|\frac sr |$ is bounded), it follows that $\deg(\phi)$ is bounded and hence there are only finitely many possibilities for $\phi$.  Thus, we can conclude that $\phi(P)$ as in \eqref{eq:thesetofthm} either lies in $Z_\epsilon$ or $P$ comes from a finite set.

As for the conclusion $\phi(P) \in Z_\epsilon$, it seems difficult to change this to some kind of a finiteness statement, even under further assumptions.  For a single $P$, $\phi(P) \in Z_\epsilon$ only occurs for finitely many $\phi$'s if we assume that the semigroup orbit of $P$ is generic, that is, any infinite subset is Zariski-dense. However, when we allow $P$ to vary, one could in theory have infinitely many $\phi(P)$'s lying in $Z_\epsilon$, coming from infinitely many $P$'s with each having generic semigroup orbit.
\end{remark}

We now prove Theorem \ref{thm:fggp2}, which deals with multiplicative dependence occurring between an iterate and its further iterate.

\begin{proof}[Proof of Theorem \ref{thm:fggp2}]
Let $(\phi\circ \psi(P), \psi(P))$ be in the set \eqref{eq:thesetofthm2}.  If $\gamma =(\gamma_m)_{m\ge 1}$ is an infinite word consisting of $\phi_1, \ldots, \phi_\ell$, Kawaguchi \cite[Theorem 3.3]{Krand} has defined the canonical height $\widehat h_\gamma$ with respect to $\gamma$ by
\[
\widehat h_\gamma(P) = \lim_{n\to \infty} \frac{h(\gamma_n\circ \cdots \circ \gamma_1(P))}{\deg (\gamma_n\circ \cdots \circ \gamma_1)},
\]
and shown that
\[
\widehat h_{(\gamma_m)_{m\ge 2}}(\gamma_1(P)) = \deg (\gamma_1) \widehat h_\gamma(P)
\]
for all $P$ and that there exists a constant $C_5$ (depending on $\phi_1, \ldots, \phi_\ell$, but independent of $\gamma$) such that
\[
|\widehat h_\gamma(P) - h(P) | \le C_5
\]
for all $P$.
Let $\gamma$ be the infinite word corresponding to $\phi \circ \phi \circ \cdots$.  Then letting $Q = \psi(P)$, whenever $\phi\circ \psi(P) \notin Z_\epsilon \cup |D|$, we must have
\begin{align}
\epsilon & \Big( (\deg D) \big(\deg(\phi) \widehat h_\gamma(Q) - C_5\big)  - C_2 \Big) \notag\\
&= \epsilon \Big( (\deg D) \big( \widehat h_\gamma(\phi(Q)) - C_5\big)  - C_2 \Big) \label{eq:needcanonht}\\
&\le \epsilon \Big( (\deg D) h(\phi(Q)) - C_2 \Big) \notag\\
&\le \epsilon h(D, \phi(Q)) \notag\\
&\le \sum_{v\notin S} \lambda_v(D, \phi(Q))\notag\\
&\le \left| \frac sr \right| \sum_{v\notin S} \lambda_v(D, Q)\notag\\
&\le \left| \frac sr \right| \left(h(D, Q) + C_3\right)\notag\\
&\le \left| \frac sr \right| \Big((\deg D) h(Q) + C_4 \Big)\notag\\
&\le \left| \frac sr \right| \Big((\deg D) \big(\widehat h_\gamma(Q) + C_5\big) + C_4 \Big), \notag
\end{align}
where we used $C_2, C_3, C_4$ as in the proof of Theorem \ref{thm:fggp}.
Rearranging terms, we must have
\begin{multline}
(\deg D) \Big(\epsilon \deg(\phi) - \left| \frac sr \right|\Big) \widehat h_\gamma(Q)\\
\le \epsilon ((\deg D) C_5 + C_2) + \left| \frac sr \right| ((\deg D) C_5 + C_4).  \label{eq:canonhtestimate}
\end{multline}
Choosing $\epsilon \ge \frac {1+c}2$, the coefficient of $\widehat h_\gamma(Q)$ in \eqref{eq:canonhtestimate} is bounded below by
\[
(\deg D) \frac {1-c}2 \deg(\phi),
\]
while the right-hand side of \eqref{eq:canonhtestimate} is a constant (note that $r$ and $s$ are fixed for this theorem).  Therefore, we conclude that $\deg(\phi) \widehat h_\gamma(Q)$ is bounded above by a constant.  Since
\[
\deg(\phi)\widehat h_\gamma(Q) = \widehat h_\gamma(\phi(Q)) \ge h(\phi(Q)) - C_5 = h(\phi(\psi(P))) - C_5,
\]
it follows that $h(\phi(\psi(P)))$ is also bounded above by a constant, and by the Northcott property, there are only finitely many such $\phi(\psi(P))$'s.  This argument also shows that $\widehat h_\gamma(Q)$ is also bounded, so we obtain the finiteness of $Q = \psi(P)$ as well.
\end{proof}

\begin{remark}
The same proof works even for an infinitely-generated $\mathcal F$, as long as the generators form a ``bounded'' set of maps, as defined in Kawaguchi \cite{Krand}.
\end{remark}

\begin{remark}\label{rem:boss}
This can be thought of as a natural generalization of B\' erczes--Ostafe--Shparlinski--Silverman \cite{BOSS}: for a single map $f$, $\iter f {(n+k)} = \iter fn \circ \iter fk$, so we can always write in the form $\phi \circ \psi$, and \eqref{eq:thesetfggp} of Theorem \ref{thm:fggp} was proved affirmatively by Silverman \cite{S1} for rational functions $f$ on $\pp^1$ whose preimages of $0$ contain more than $2$ points.  The condition in \cite{BOSS} is
\[
n\ge \frac{\log |s/r|}{\log d} + 1,
\]
and for $\phi = \iter fn$, this is equivalent to $\frac 1d \deg (\phi) \ge |s/r|$.
\end{remark}

\begin{remark}\label{rem:canonht}
We use the canonical height, instead of the naive height, in the proof of Theorem \ref{thm:fggp2} because of step \eqref{eq:needcanonht}.  If we use the naive height, we only have \eqref{eq:usualhtineq}, so it will be necessary to adjust by a constant multiple of $\deg \phi$ in \eqref{eq:needcanonht} and the argument does not work.  Similarly, since the value of the canonical height $\widehat h_\gamma$ depends on $\gamma$, the argument of the proof of Theorem \ref{thm:fggp2} does not immediately generalize to the general orbit dependency of $\phi(P)^r = u\cdot \psi(P)^s$.
\end{remark}

\section{Regarding Hypothesis \eqref{eq:thesetfggp} and Vojta's conjecture}\label{sec:vojta}
We will first recall the Main Conjecture of Vojta \cite[Conjecture 3.4.3]{V2}.

\begin{conjecture}\label{conj:vojta}
Let $X$ be a smooth projective variety defined over a number field $k$, $K$ be a canonical divisor, $A$ be an ample divisor, and $D$ be a reduced normal-crossings divisor.  Let $S$ be a finite subset of $M_k$.  Then given $\epsilon >0$, there exists a Zariski-closed $Z = Z_\epsilon \neq X$ such that
\begin{equation}\label{eq:vojtacon}
\sum_{v\in S} \lambda_v(D, P) + h(K,P) \le \epsilon h(A, P)
\end{equation}
for all $P\in (X\setminus Z)(k)$.
\end{conjecture}

One of the known cases of the conjecture above is when $X$ is the projective space and $D$ is a union of hyperplanes, known as Schmidt subspace theorem.
We now use Conjecture \ref{conj:vojta} to obtain paucity of integral points in orbits, confirming hypothesis \eqref{eq:thesetfggp} for sufficiently small $\epsilon$:

\begin{theorem}\label{thm:basedonvojta}
Let $X$ be a smooth projective variety defined over a number field $k$, $K$ be a canonical divisor of $X$, $A$ an ample divisor of $X$, $S$ be a finite set of places of $k$, and $\phi_1, \ldots, \phi_\ell$ be endomorphisms of $X$ defined over $k$. Let $\mathcal F$ be the semi-group generated by $\phi_1, \ldots, \phi_\ell$ under composition.  Suppose that there exist $\psi_1, \ldots, \psi_m\in \mathcal F$ and a nontrivial effective divisor $D$ defined over $k$ such that
\begin{itemize}
\item[\emph{(i)}] $\psi_i^*(D) = D_i^{(\text{nc})} + D_i^{(\text{rest})}$ for an effective divisor $D_i^{(\text{rest})}$
\item[\emph{(ii)}] $D_i^{(\text{nc})}$ is a reduced normal-crossings divisor
\item[\emph{(iii)}] $D_i^{(\text{nc})} + K$  is big
\item[\emph{(iv)}] $\displaystyle \mathcal F\setminus \{\id\}  = \bigcup_{i=1}^m \psi_i \circ \mathcal F$.
\end{itemize}
Then assuming Vojta's conjecture for $(X, D_i^{(\text{nc})})$, for all sufficiently small $\epsilon$, the set
\begin{equation}\label{eq:theset_vojta}
\big\{\phi(P): P\in X(k), \phi\in \mathcal F\setminus \{\id\}, \sum_{v\notin S} \lambda_v(D, \phi(P)) \le \epsilon h(A, \phi(P))\big\}
\end{equation}
is Zariski-non-dense in $X$.  If (iv) is replaced by
\[
\displaystyle \mathcal F \setminus \left(\{\id\} \cup \bigcup_{i=1}^m \psi_i \circ \mathcal F\right) \text{ is finite, }
\]
the set \eqref{eq:theset_vojta} is contained in the union of a Zariski-non-dense set with the images of $k$-rational points by these finitely many maps.
\end{theorem}

From the proof, it will become clear how small $\epsilon$ has to be, and we will explore this in Example \ref{example}.  As noted in the introduction, a similar result was obtained by the first author \cite[Theorem 5.1]{M3} using Evertse--Ferretti result \cite{EF} in place of Vojta's conjecture. Note that in one-dimension, the set of $P\in\pp^1(k)$ for which $\phi(P)$ is quasi-integral is finite by Siegel's theorem as long as $\phi$ has at least three poles.  As a result, for $\pp^1$, we conclude that \eqref{eq:theset_vojta} is finite as long
\[
\mathcal F \setminus \left(\{\id\} \cup \bigcup_{i=1}^m \psi_i \circ \mathcal F\right)
\]
is finite and each rational function in this set has at least three poles.  An analog of Siegel's theorem is not unconditionally known in higher-dimensions, and we typically have to assume conditions like (i)--(iii) even under Vojta's conjecture.

\begin{proof}[Proof of Theorem \ref{thm:basedonvojta}]
Since $D_i^{(\text{nc})}+K$ is big, there exists a proper Zariski-closed set $Z_1$ and a positive constant $c_i'$ such that
\begin{equation}\label{eq2}
h(D_i^{(\text{nc})} + K, Q) \ge c_i' h(A, Q) - O(1)
\end{equation}
for all $Q\notin Z_1$, where we use Landau's big $O$ notation.
Now suppose $\phi(P)$ belongs to the set \eqref{eq:theset_vojta} for some $\phi\in \mathcal F$.  We may assume that $\phi = \psi_i \circ \eta$ for some $\eta \in \mathcal F$ and $i\le m$.  By the functoriality of heights,
\begin{align}
\sum_{v\notin S} &\lambda_v(D, \phi(P)) \sim h(D, \psi_i\circ \eta(P)) - \sum_{v\in S} \lambda_v(D,\psi_i\circ  \eta(P)) \notag\\
&\sim h(\psi_i^*(D), \eta(P))  - \sum_{v\in S} \lambda_v(\psi_i^*(D), \eta(P)) \notag\\
&\sim \sum_{v\notin S} \lambda_v(\psi_i^*(D), \eta(P)) \label{eq1}
\end{align}
where we use $\sim$ to indicate that the difference of both sides is a  bounded function away from a Zariski-non-dense subset of $X(k)$.  In addition, because of ampleness, there exist $c_i''$ such that
\begin{equation}\label{eq3}
h(\psi_i^*(A), Q) \le  c_i'' h(A, Q) + O(1)
\end{equation}
for all $Q$, so we have
\begin{equation}\label{eq4}
h(A, \psi_i\circ \eta(P)) \sim h(\psi_i^*(A), \eta(P)) \le c_i'' h(A, \eta(P)) + O(1).
\end{equation}
Applying \eqref{eq2} with $Q = \eta(P)$ and combining with \eqref{eq1}, \eqref{eq4}, and the condition in \eqref{eq:theset_vojta}, we have
\begin{align*}
\epsilon & c_i'' h(A, \eta(P))\\
&\ge \epsilon h(\psi_i^*(A), \eta(P)) - O(1)\\
&\ge \epsilon h(A, \psi_i \circ \eta(P)) - O(1)\\
&\ge \sum_{v\notin S} \lambda_v(D, \psi_i \circ \eta(P)) - O(1)\\
&\ge \sum_{v\notin S} \lambda_v(\psi_i^*(D), \eta(P)) - O(1)\\
&\ge \sum_{v\notin S} \lambda_v(D_i^{(\text{nc})}, \eta(P)) - O(1),
\end{align*}
so
\begin{align}
&\sum_{v\in S} \lambda_v(D_i^{(\text{nc})}, \eta(P)) + h(K, \eta(P)) \sim h(D_i^{(\text{nc})}+K, \eta(P)) - \sum_{v\notin S} \lambda_v(D_i^{(\text{nc})}, \eta(P))\notag\\
&\ge c_i' h(A, \eta(P)) - \epsilon c_i'' h(A, \eta(P)) - O(1).\notag
\end{align}
Therefore, whenever $\epsilon$ is small enough so that $c_i' - \epsilon c_i'' >0$ for all $i$, Vojta's conjecture applied to the divisor $D_i^{(\text{nc})}$ implies that there exists a proper Zariski-closed set $Z_2$ such that $\eta(P) \in Z_2$.  This argument shows that if $\psi_i(\eta(P))$ is in the set \eqref{eq:theset_vojta}, $\psi_i(\eta(P)) \in \psi_i(Z_1) \cup \psi_i(Z_2)$.
\end{proof}

\begin{remark}\label{remark:integralpoint}
Since the set of $(D,S)$-integral points satisfy $\sum_{v\notin S} \lambda_v(D, P) \le C$ for some constant $C$, this theorem shows Zariski-non-density of integral points in semi-group orbit.  More generally, for quasi-$(D,S)$-integral points, namely, for $P$ satisfying
\[
\sum_{v\notin S} \lambda_v(D, P) < \epsilon ' h(D, P)
\]
then $A$ being ample means that there exists a constant $C$ such that $h(D, P) \le C h(A, P)$ for all $P$, so we must have
\[
\sum_{v\notin S} \lambda_v(D, P) < \epsilon ' C h(A, P).
\]
Hence for $\epsilon ' = \frac {\epsilon}C$, such $P$ lies in \eqref{eq:theset_vojta}.
\end{remark}

We include a quick unconditional example of Theorem \ref{thm:basedonvojta} together with its implication to Theorem \ref{thm:fggp}.

\begin{example}\label{example}
Let $\phi_1 = [F_0: \cdots : F_{N-1} : L_1 \cdots L_d]$ and $\phi_2 = [G_0: \cdots : G_{N-1} : M_1 \cdots M_d]$ be endomorphisms on $\pp^N$, where $F_i$ and $G_i$ are homogeneous polynomials of degree $d$, and where $L_1, \ldots, L_d$ and $M_1, \ldots, M_d$ are linear forms in general position, respectively.  If $d > N+1$ then for $D = (X_N = 0)$, $c_i'$ in the proof of Theorem \ref{thm:basedonvojta} can be taken to be $d-(N+1)$ while $c_i''$ to be $d$.  Therefore, if $d$ is large enough so that $\frac{d-(N+1)}d > \frac{1+c}2$, we have \eqref{eq:thesetfggp} unconditionally for sufficiently large $\epsilon$ by Theorem \ref{thm:basedonvojta} using Schmidt subspace theorem in place of Vojta's conjecture, and it follows that the set \eqref{eq:thesetofthm} and the set \eqref{eq:thesetofthm2} are Zariski-non-dense from Theorems \ref{thm:fggp} and \ref{thm:fggp2} respectively.

Similarly, assume that $\phi_1 = [X_1 \cdots X_N F_0: F_1 : \cdots : F_{N}]$ and $\phi_2 = [G_0: \cdots : G_{N}]$ are endomorphisms on $\pp^N$ with $F_0$ of degree $d-N$, $F_1, \ldots, F_N$ are each product of $d$ linear forms so that their total union is in general position, and where $G_1, \ldots, G_{N}$ are each products of $e$ linear forms whose total union is in general position and $G_0$ is a product of $e$ linear forms in general position. Then $(\iter {\phi_1} 2)^*(X_0 = 0)$ contains the union of $dN$ hyperplanes in general position, $(\phi_1 \circ \phi_2)^*(X_0 = 0)$ contains the union of $eN$ hyperplanes in general position, and $\phi_2^*(X_0=0)$ is the union of  $e$ hyperplanes in general position.  Since
\[
\mathcal F \setminus \{\id\} = \{\phi_1\} \cup (\iter {\phi_1} 2\circ \mathcal F) \cup (\phi_1 \circ \phi_2\circ \mathcal F) \cup (\phi_2 \circ \mathcal F),
\]
as long as $dN$ and $e$ are sufficiently large (so this may work with a smaller $d$ than the situation of the previous paragraph) and integral points in $\phi_1(\pp^N(k))$ are known to be Zariski-non-dense, \eqref{eq:thesetfggp} holds unconditionally for a large enough $\epsilon$ so that the set \eqref{eq:thesetofthm} and the set \eqref{eq:thesetofthm2} are both Zariski-non-dense.
\end{example}

\textbf{Acknowledgements}: This project was initiated during the conference ``Complex Dynamics and Related Topics" at RIMS of Kyoto University and continued during the AIM workshop ``Dynamics of Multiple Maps." Jorge Mello was supported by the ARC Grant DP180100201 and Oakland University FCT and BFA grants, while Yu Yasufuku is supported by JSPS Kakenhi 19K03412, 24K06696, and 25H00587.

\bibliographystyle{amsplain}
\bibliography{multdep}

\end{document}